\documentclass{mfat}
\pagespan{375}{386}

\usepackage{mmap}
\usepackage[utf8]{inputenc}

\newtheorem{proposition}{Proposition}

\theoremstyle{remark}

\theoremstyle{definition}

\begin{document}

\title[A criterion for continuity in a parameter of solutions]
{A criterion for continuity in a parameter of solutions to generic boundary-value problems for
higher-order differential systems}


\author[V. Mikhailets]{Vladimir Mikhailets}
\address{Institute of Mathematics, National Academy of
Sciences of Ukraine, 3 Teresh\-chenkivs'ka, Kyiv, 01601, Ukraine;
 National Technical University of Ukraine "Kyiv Polytechnic Institute", 37 Prospect Peremogy, Kyiv, 03056, Ukraine}
\email{mikhailets@imath.kiev.ua}


\author[A. Murach]{Aleksandr Murach}
\address{Institute of Mathematics, National Academy of Sciences of Ukraine,
3 Teresh\-chenkivs'ka, Kyiv, 01601, Ukraine; Chernihiv National
Pedagogical University, 53 Het'mana Polubotka,
 Chernihiv, 14013, Ukraine}
 \email{murach@imath.kiev.ua}


\author[V. Soldatov]{Vitalii Soldatov}
\address{Institute of Mathematics, National Academy of Sciences of Ukraine,
3 Teresh\-chenkivs'ka, Kyiv, 01601, Ukraine}
\email{soldatovvo@ukr.net}

\date{15/08/2016}

\begin{abstract}
We consider the most general class of linear boundary-value problems for ordinary differential systems,
of order $r\geq1$, whose solutions belong to the complex space $C^{(n+r)}$, with $0\leq n\in\mathbb{Z}$.
The boundary conditions can contain derivatives of order $l$, with $r\leq l\leq n+r$, of the solutions.
We obtain a constructive criterion under which the solutions to these problems are continuous with respect
to the parameter in the normed space $C^{(n+r)}$. We also obtain a two-sided estimate for the degree
of convergence of these solutions.
\end{abstract}

\maketitle

\section{Introduction}\label{1}

Questions concerning the validity of passage to the limit in
parameter-dependent Cauchy problems and boundary-value problems
are very important in the theory of ordinary differential
equations. Fundamental results on the continuous in the parameter
of solutions to the Cauchy problem for nonlinear systems were
obtained by Gikhman \cite{Gikhman1952}, Krasnosel'skii and
S.~Krein \cite{KrasnoselskiiKrein1955}, Kurzweil and Vorel
\cite{KurzweilVorel1957CMJ}. These results were refined and
supplemented for linear systems by Levin \cite{Levin1967dan},
Opial \cite{Opial1967}, Reid \cite{Reid1967}, and Nguyen
\cite{NguenTheHoan1993}.

Parameter-dependent boundary-value problems are far less studied than
the Cauchy problem, which is connected with great diversity of
boundary conditions. Kiguradze \cite{Kiguradze1975, Kiguradze1987,
  Kiguradze2003} and Ashordia \cite{Ashordia1996} introduced and
studied a class of general linear boundary-value problems for
systems of first-order ordinary differential equations. The solutions
$y$ to these problems are supposed to be absolutely continuous on a
compact interval $[a,b]$, and the boundary condition is given in the
form $By=q$, where $B:C([a,b],\mathbb{R}^{m})\to\mathbb{R}^{m}$ is an
arbitrary continuous linear operator ($m$ is the number of
differential equations in the system). Kiguradze and Ashordia obtained
conditions under which the solutions to these problems are continuous
in the normed space $C([a,b],\mathbb{R}^{m})$ with respect to the
parameter.  Recently \cite{KodliukMikhailetsReva2013,
  MikhailetsChekhanova2015} these results were refined and extended to
complex-valued functions and systems of higher-order differential
equations.

Other broad classes of linear boundary-value problems were introduced and investigated in \cite{MikhailetsReva2008DAN8, KodliukMikhailets2013JMS, GnypKodlyukMikhailets2015UMJ} and \cite{MikhailetsChekhanova2014DAN7, Soldatov2015UMJ}. These classes  relate closely to the classical scales of the function Sobolev spaces
and the spaces $C^{(l)}$ of continuously differential functions, respectively. The boundary conditions for the problems from these classes are given in the form $By=q$, where $B$ is an arbitrary continuous linear operator that acts from the corresponding function space to the finite-dimensional complex space. It is naturally to call these boundary-value problems generic with respect to the indicated function space. Constructive sufficient conditions for continuity of their solutions in the parameter in the corresponding spaces are obtained. These results were applied to the investigation of multipoint boundary-value problems \cite{Kodliuk2012Dop11}, Green's matrices of boundary-value problems \cite{KodliukMikhailetsReva2013, MikhailetsChekhanova2015}, in the spectral theory of differential operators with distributional coefficients \cite{GoriunovMikhailets2010MN87.2, GoriunovMikhailetsPankrashkin2013EJDE, GoriunovMikhailets2012UMJ63.9}.

In the present paper, we investigate the class of generic boundary-value problems for differential systems of order $r\geq1$ with respect to the complex space $C^{(n+r)}$, where the integer $n\geq0$. This class was introduced in \cite{MikhailetsChekhanova2014DAN7} for $r=1$ and
in \cite{Soldatov2015UMJ} for $r\geq2$. We assume that all coefficients and right-hand sides of these systems belong to the space $C^{(n)}$. Since the solution $z$ to such a system runs through the whole space $(C^{(n+r)})^{m}$ provided that its right-hand side runs through the whole space $(C^{(n)})^m$, the boundary condition is given in the most general form $Bz=q$ where $B:(C^{(n+r)})^{m}\to\mathbb{C}^{rm}$ is an arbitrary continuous linear operator and $q\in\mathbb{C}^{rm}$. This condition covers all the classical types of boundary conditions such as initial conditions in the Cauchy problem, multipoint conditions, integral conditions, and also nonclassical conditions containing the derivatives $z^{(l)}$ with $r\leq l\leq n+r$ of the unknown function.

The sited papers \cite{MikhailetsChekhanova2014DAN7, Soldatov2015UMJ}  provide a constructive sufficient conditions under which the solution to an arbitrary parameter-dependent boundary-value problem from this class depends continuously on the parameter in the normed space $(C^{(n+r)})^{m}$. The goal of the present paper is to show that these conditions are also necessary. Thus, we will prove a criterion for  continuity in the parameter of the solution to this problem in the indicated space. Besides, we will obtain a two-sided estimate for the degree of convergence of the solutions.

The approach developed in this paper can be applied to generic boundary-value problems with respect to other function spaces \cite{GnypMikhailetsMurach2016, MikhailetsMurachSoldatov2016}.

\section{Main results}\label{6sec2}
We arbitrarily choose a compact interval $[a,b]\subset\mathbb{R}$ and integers $n\geq0$, $r\geq1$, and $m\geq1$. We use the Banach spaces
\begin{equation}\label{6.spaces}
(C^{(l)})^{m}:=C^{(l)}([a,b],\mathbb{C}^{m})\quad\mbox{and}
\quad(C^{(l)})^{m\times m}:=C^{(l)}([a,b],\mathbb{C}^{m\times m}),
\end{equation}
with $0\leq l\in\mathbb{Z}$. They consist respectively of all vector-valued functions and $(m\times m)$-matrix-valued functions whose components belong to the space $C^{(l)}:=C^{(l)}([a,b],\mathbb{C})$ of all $l$ times continuously differentiable functions $x:[a,b]\rightarrow\mathbb{C}$. This space is endowed with the norm
$$
\|x\|_{(l)}:=\sum_{j=0}^{l}
\max\bigl\{|x^{(j)}(t)|:\,t\in[a,b]\bigr\}.
$$
The norms in the spaces \eqref{6.spaces} are equal to the sum of the norms in $C^{(l)}$ of all components of the vector-valued or matrix-valued functions and are also denoted by $\|\cdot\|_{(l)}$. It will be always clear from the context to which space (scalar or vector-valued or matrix-valued functions) the norm $\|\cdot\|_{(l)}$ relates.

Let $\varepsilon_0>0$. We consider the following linear boundary-value problem depending on the parameter $\varepsilon\in[0,\varepsilon_0)$:
\begin{gather}\label{6.syste}
L(\varepsilon)z(t,\varepsilon)\equiv z^{(r)}(t,\varepsilon)+
\sum_{j=1}^{r}K_{r-j}(t,\varepsilon)z^{(r-j)}(t,\varepsilon)=
f(t,\varepsilon),\quad a\leq t\leq b,\\
B(\varepsilon)z(\cdot,\varepsilon)=q(\varepsilon).\label{6.kue}
\end{gather}
We suppose for every $\varepsilon\in[0,\varepsilon_0)$ that
$$z(\cdot,\varepsilon)\in(C^{(n+r)})^{m}$$ is an unknown
vector-valued function, whereas the matrix-valued functions
$K_{r-j}(\cdot,\varepsilon)\in(C^{(n)})^{m\times m}$, with
$j\in\{1,\ldots,r\}$, vector-valued function
$f(\cdot,\varepsilon)\in(C^{(n)})^m$, continuous linear operator
\begin{equation}\label{6.Be}
B(\varepsilon):(C^{(n+r)})^m \to \mathbb{C}^{rm}
\end{equation}
and vector $q(\varepsilon)\in \mathbb{C}^{rm}$ are arbitrarily given. In the paper, we interpret vectors as columns.

With this problem, we associate the continuous linear operator
\begin{equation}\label{6.LBe}
(L(\varepsilon),B(\varepsilon)):(C^{(n+r)})^{m}\rightarrow (C^{(n)})^{m}\times\mathbb{C}^{rm}.
\end{equation}
The operator \eqref{6.LBe} is Fredholm with index zero for every
$\varepsilon\in[0,\varepsilon_0)$ because  this operator is a
finite-dimensional perturbation of the isomorphism that acts
between the spaces $(C^{(n+r)})^{m}$ and
$(C^{(n)})^{m}\times\mathbb{C}^{rm}$ and corresponds to the Cauchy
problem for the differential system \eqref{6.syste} (see, e.g.,
\cite{Hermander85}).

For the boundary-value problem \eqref{6.syste}, \eqref{6.kue}, we consider the following three

\medskip

\noindent\textbf{Limit Conditions} as $\varepsilon\to0+$:
\begin{itemize}
  \item [(I)] $K_{r-j}(\cdot,\varepsilon)\to K_{r-j}(\cdot,0)$ in $(C^{(n)})^{m\times m}$ for each $j\in\{1,\ldots,r\}$;
  \item [(II)] $B(\varepsilon)z\to B(0)z$ in $\mathbb{C}^{rm}$ for every $z\in(C^{(n+r)})^m$;
  \item [(III)] $f(\cdot,\varepsilon)\to f(\cdot,0)$ in $(C^{(n)})^{m}$, and $q(\varepsilon)\to q(0)$ in $\mathbb{C}^{rm}$.
\end{itemize}

\smallskip

We also consider

\medskip

\noindent\textbf{Condition (0).} The limiting homogeneous boundary-value problem
\begin{equation}\label{6.LB0}
L(0)\,z(t,0)=0,\quad a\leq t\leq b,\quad\mbox{and}\quad B(0)\,z(\cdot,0)=0
\end{equation}
has only the trivial solution.

\medskip

Note that a criterion for the problem \eqref{6.LB0} satisfy Condition~(0) is given in \cite[Theorem~2]{Soldatov2015UMJ}. This criterion is formulated in terms of the matrices $K_{r-j}(\cdot,0)$ and the operator~$B(0)$.

Now we introduce our

\medskip

\noindent\textbf{Basic Definition.} We say that the solution to the boundary-value problem \eqref{6.syste}, \eqref{6.kue} is continuous in the parameter $\varepsilon$ at $\varepsilon=0$ if the following two conditions are fulfilled:
\begin{itemize}
\item [$(\ast)$] There exists a positive number $\varepsilon_{1}<\varepsilon_{0}$ such that for an arbitrary number $\varepsilon\in\nobreak[0,\varepsilon_{1})$, function $f(\cdot,\varepsilon)\in(C^{(n)})^m$, and vector $q(\varepsilon)\in\mathbb{C}^{rm}$ this problem has a unique solution $z(\cdot,\varepsilon)\in(C^{(n+r)})^{m}$.
\item [$(\ast\ast)$] Limit Condition (III) implies the convergence $z(\cdot,\varepsilon)\to z(\cdot,0)$ in $(C^{(n+r)})^{m}$ as $\varepsilon\to0+$.
\end{itemize}

\noindent\textbf{Main Theorem.} \it The solution to the boundary-value problem \eqref{6.syste}, \eqref{6.kue} is continuous in the parameter $\varepsilon$ at $\varepsilon=0$ if and only if this problem satisfies Condition~\rm(0) \it and Limit Conditions \rm (I) \it and \rm (II)\it.\rm

\medskip

We supplement this result by the following theorem. Let
\begin{equation*}
d_{n}(\varepsilon):=
\|L(\varepsilon)z(\cdot,0)-f(\cdot,\varepsilon)\|_{(n)}+
\|B(\varepsilon)z(\cdot,0)-q(\varepsilon)\|_{\mathbb{C}^{rm}}.
\end{equation*}

\medskip

\noindent\textbf{Estimation Theorem.} \it
Suppose that the boundary-value problem \eqref{6.syste}, \eqref{6.kue} satisfies Condition~\rm(0) \it and Limit Conditions \rm (I) \it and \rm (II)\it. Then there exist positive numbers $\varepsilon_{2}<\varepsilon_{1}$ and $\varkappa_{1}$, $\varkappa_{2}$ such that for every $\varepsilon\in(0,\varepsilon_{2})$ the following two-sided estimate holds:
\begin{equation}\label{6.bound}
\varkappa_{1}\,d_{n}(\varepsilon)\leq
\|z(\cdot,0)-z(\cdot,\varepsilon)\|_{(n+r)}
\leq\varkappa_{2}\,d_{n}(\varepsilon).
\end{equation}
Here, the numbers $\varepsilon_{2}$, $\varkappa_{1}$, and $\varkappa_{2}$ do not depend on $z(\cdot,0)$ and $z(\cdot,\varepsilon)$. \rm

\medskip

According to \eqref{6.bound}, the error and discrepancy of the solution $z(\cdot,\varepsilon)$ to the boundary-value problem \eqref{6.syste}, \eqref{6.kue} are of the same degree provided that we consider $z(\cdot,0)$ as an approximate solution to this problem.

We will prove these theorems in the next section.

Note that every linear continuous operator \eqref{6.Be} can be
uniquely represented in the form
\begin{equation}\label{Be-description}
B(\varepsilon)z=
\sum_{k=1}^{n+r}\beta_{k}(\varepsilon)\,z^{(k-1)}(a)+
\int_a^b(d\Phi(t,\varepsilon))z^{(n+r)}(t),\quad\mbox{with}\quad z\in(C^{(n+r)})^m.
\end{equation}
Here, each $\beta_{k}(\varepsilon)$ is a number $rm\times m$-matrix, and $\Phi(\cdot,\varepsilon)$ is an
$rm\times m$-matrix-valued function formed by scalar functions that are of bounded variation on $[a,b]$,
right-continuous on $(a,b)$, and equal to zero at $t=a$. (Certainly, the integral is understood in the
Riemann-Stieltjes sense.) Representation \eqref{Be-description} follows from the known description of the dual of  $C^{(n+r)}$; see, e.g., \cite[p.~344]{DanfordShvarts1958}. Using this representation, we can reformulate Limit Condition~(II) in an explicit form. Namely, Limit Condition~(II) is equivalent to that the following four conditions are fulfilled as $\varepsilon\to0+$:
\begin{itemize}
  \item [(2a)] $\beta_{k}(\varepsilon)\to\beta_{k}(0)$ for every $k\in\{1,\ldots,n+r\}$;
  \item [(2b)] $\|V_{a}^{b}\Phi(\cdot,\varepsilon)\|_{\mathbb{C}^{rm\times m}}=O(1)$;
  \item [(2c)] $\Phi(b,\varepsilon)\to\Phi(b,0)$;
  \item [(2d)] $\int_{a}^{t}\Phi(s,\varepsilon)ds\to
\int_{a}^{t}\Phi(s,0)ds$ for every $t\in(a,b]$.
\end{itemize}
(Of course, the convergence is considered in $\mathbb{C}^{rm\times m}$.) This equivalence follows from the F.~Riesz criterion of the weak convergence of linear continuous functionals on $C^{(0)}$ (see, e.g., \cite[Ch.~III, Sect.~55]{RieszSz-Nagy56}). It is useful to compare these conditions with the criterion of the norm convergence of the operators $B(\varepsilon)\to B(0)$ as $\varepsilon\to0+$. It asserts that this norm convergence is equivalent to the fulfilment of condition (2a) and
\begin{equation}\label{6.conv-in-variation}
V_{a}^{b}\bigl(\Phi(\cdot,\varepsilon)-\Phi(\cdot,0)\bigr)\to0
\end{equation}
as $\varepsilon\to0+$. Condition \eqref{6.conv-in-variation} is much stronger than the system of conditions (2b)--(2d). Indeed, condition~\eqref{6.conv-in-variation} implies the uniform convergence of the functions $\Phi(t,\varepsilon)$ to $\Phi(t,0)$ on $[a,b]$ as $\varepsilon\to0+$ whereas conditions (2b)--(2d) do not entail the pointwise convergence of these functions at least in one point of $(a,b)$.

\section{Proofs}\label{6sec3}

\begin{proof}[Proof of Main Theorem.]
The sufficiency of Conditions (0), (I) and~(II) for problem \eqref{6.syste}, \eqref{6.kue} to satisfy Basic definition is proved in \cite[Theorem~3]{Soldatov2015UMJ}. Let us prove the necessity. Assume that this problem satisfies Basic definition. Then Condition~(0) is fulfilled. It remains to prove that this problem satisfies Conditions (I) and~(II). We divide our reasoning into three steps.

\emph{Step~$1$.} We will prove here that the boundary-value problem \eqref{6.syste}, \eqref{6.kue} satisfies Limit Condition~(I). In the $r\geq2$ case, we previously reduce this problem to a boundary-value problem for first-order differential system. To this end we put, as usual,
\begin{gather*}
y(\cdot,\varepsilon):=\mathrm{col}\bigl(z(\cdot,\varepsilon),z'(\cdot, \varepsilon),\ldots,z^{(r-1)}(\cdot, \varepsilon)\bigr)\in(C^{(n+1)})^{rm},\\
g(\cdot,\varepsilon):=\mathrm{col}\bigl(0,f(\cdot, \varepsilon)\bigr)
\in(C^{(n)})^{rm},
\end{gather*}
and
\begin{equation*}
A(\cdot,\varepsilon):=\left(
\begin{array}{ccccc}
O_m & I_m & O_m & \ldots & O_m \\
O_m & O_m & I_m & \ldots & O_m \\
\vdots & \vdots & \vdots & \ddots & \vdots \\
O_m & O_m & O_m & \ldots & I_m \\
K_0(\cdot,\varepsilon) & K_1(\cdot,\varepsilon) & K_2(\cdot,\varepsilon) & \ldots & K_{r-1}(\cdot,\varepsilon)\\
\end{array}\right)\in(C^{(n)})^{rm\times rm},
\end{equation*}
with $O_m$ and $I_m$ denoting zero and identity $(m \times m)$-matrices, respectively. In view of representation \eqref{Be-description}, we also put
\begin{equation}\label{6.Be_fo}
N(\varepsilon)y:=\sum_{k=1}^{r-1}\beta_{k}(\varepsilon)y_{k}(a)+
\sum_{k=r}^{n+r}\beta_{k}(\varepsilon)y_{r}^{(k-r)}(a)
+\int_{a}^{b}(d\Phi(t,\varepsilon))y_{r}^{(n+1)}(t)
\end{equation}
for an arbitrary vector-valued function $y=\mathrm{col}(y_{1},\ldots,y_{r})$ with $y_{1},\ldots,y_{r}\in(C^{(n+1)})^{m}$. The linear mapping $y\mapsto Ny$ acts continuously from $(C^{(n+1)})^{rm}$ to $\mathbb{C}^{rm}$.

Evidently, a function $z(\cdot,\varepsilon)\in(C^{(n+r)})^{m}$ is a solution to the boundary-value problem \eqref{6.syste}, \eqref{6.kue} if and only if the function $y(\cdot,\varepsilon)$ is a solution to the boundary-value problem
\begin{gather}\label{6.fo}
y'(t,\varepsilon)+A(t,\varepsilon)y(t,\varepsilon)=g(t,\varepsilon),\quad a\leq t\leq b,\\
N(\varepsilon)y(\cdot,\varepsilon)=q(\varepsilon). \label{6.BC_fo}
\end{gather}
In the $r=1$ case, we put $y(\cdot,\varepsilon):=z(\cdot,\varepsilon)$, $g(\cdot,\varepsilon):=f(\cdot,\varepsilon)$,  $A(\cdot,\varepsilon):=A_{0}(\cdot,\varepsilon)$, and $N(\varepsilon):=B(\varepsilon)$ for the sake of uniformity in notation on Step~1, so the problem \eqref{6.syste}, \eqref{6.kue} coincides with the problem \eqref{6.fo}, \eqref{6.BC_fo}.

According to condition~($\ast$) of Basic Definition, the latter
problem has a unique solution $y(t,\varepsilon)$ for every
$\varepsilon\in[0,\varepsilon_{1})$. Limit Condition~(I) is
equivalent to that $A(\cdot,\varepsilon)\to A(\cdot,0)$ in
$(C^{(n)})^{rm\times rm}$ as $\varepsilon\to0+$. Let us prove this
convergence.

We previously note the following. If $g(\cdot,\varepsilon)$ and
$q(\varepsilon)$ are independent of $\varepsilon\in[0,\varepsilon_{1})$, then $z(\cdot,\varepsilon)\to z(\cdot,0)$ in $(C^{(n+r)})^{m}$ as $\varepsilon\to0+$ by condition~($\ast\ast$) of Basic Definition. The latter convergence is equivalent to that $y(\cdot,\varepsilon)\to y(\cdot,0)$ in $(C^{(n+1)})^{rm}$ as $\varepsilon\to0+$.

Given $\varepsilon\in[0,\varepsilon_{1})$, we now consider the matrix boundary-value problem
\begin{gather}\label{6.matrix-eq}
Y'(t,\varepsilon)+A(t,\varepsilon)Y(t,\varepsilon)=0_{rm},
\quad a\leq t\leq b,\\
[N(\varepsilon)Y(\cdot,\varepsilon)]=I_{rm} \label{6.matrix-bound-cond}.
\end{gather}
Here, the unknown matrix-valued function $Y(\cdot,\varepsilon):=(y_{j,k}(\cdot,\varepsilon))_{j,k=1}^{rm}$ belongs to the space $(C^{(n+1)})^{rm\times rm}$, and we put
$$
[N(\varepsilon)Y(\cdot,\varepsilon)]:=\left(N(\varepsilon)
\begin{pmatrix}
y_{1,1}(\cdot,\varepsilon)\\
\vdots \\
y_{rm,1}(\cdot,\varepsilon)\\
\end{pmatrix}
\;\ldots\;
N(\varepsilon)\begin{pmatrix}
y_{1,rm}(\cdot,\varepsilon)\\
\vdots \\
y_{rm,rm}(\cdot,\varepsilon)\\
\end{pmatrix}\right).
$$
This problem is a union of $rm$ boundary-value problems \eqref{6.fo}, \eqref{6.BC_fo} whose right-hand sides are independent of $\varepsilon$. So, it has a unique solution $Y(\cdot,\varepsilon)\in(C^{(n+1)})^{rm\times rm}$, and $Y(\cdot,\varepsilon)\to Y(\cdot,0)$ in $(C^{(n+1)})^{rm\times rm}$ as $\varepsilon\to0+$. Besides, $\det Y(t,\varepsilon)\neq0$ for every $t\in[a,b]$. Indeed, if this were wrong, the function columns of the matrix $Y(\cdot,\varepsilon)$ would be linearly dependent, which would contradict the boundary condition \eqref{6.matrix-bound-cond}. Thus, we obtain the required convergence
\begin{equation*}
A(\cdot,\varepsilon)=-Y'(\cdot,\varepsilon)(Y(\cdot,\varepsilon))^{-1}\to
-Y'(\cdot,0)(Y(\cdot,0))^{-1}=A(\cdot,0)
\end{equation*}
in $(C^{(n)})^{rm\times rm}$ as $\varepsilon\to0+$. It implies that
\begin{equation}\label{6.bound-norm-A}
\|K_{r-j}(\varepsilon)\|_{(n)}=O(1)\quad\mbox{as}\quad\varepsilon\to0+
\quad\mbox{for each}\quad j\in\{1,\ldots,r\}.
\end{equation}

\emph{Step~$2$.} Let us prove that $\|B(\varepsilon)\|=O(1)$ as $\varepsilon\to0+$; here, $\|\cdot\|$ denotes the norm of a bounded operator from $(C^{(n+r)})^m$ to $\mathbb{C}^{rm}$. Suppose the contrary; then there exists a sequence $(\varepsilon^{(k)})_{k=1}^{\infty}\subset(0,\varepsilon_{1})$ such that $\varepsilon^{(k)}\to0$ and $0<\|B(\varepsilon^{(k)})\|\to\infty$ as $\varepsilon\to0+$. For every integer $k\geq1$, we can choose a function $w_{k}\in(C^{(n+r)})^{m}$ such that $\|w_{k}\|_{(n+r)}=1$ and $\|B(\varepsilon^{(k)})w_{k}\|_{\mathbb{C}^{rm}}\geq
\|B(\varepsilon^{(k)})\|/2$. We now let $z(\cdot,\varepsilon^{(k)}):=
\|B(\varepsilon^{(k)})\|^{-1}\,w_{k}$, $f(\cdot,\varepsilon^{(k)}):=
L(\varepsilon^{(k)})\,z(\cdot,\varepsilon^{(k)})$, and $q(\varepsilon^{(k)}):=B(\varepsilon^{(k)})\,z(\cdot,\varepsilon^{(k)})$.
Since $z(\cdot,\varepsilon^{(k)})\to0$ in $(C^{(n+r)})^{m}$ as $k\to\infty$, we get $f(\cdot,\varepsilon^{(k)})\to0$ in $(C^{(n)})^{m}$ because the boundary-value problem \eqref{6.syste}, \eqref{6.kue} satisfies Limit Condition~(I) as we have proved on step~1. Since $1/2\leq\|q(\varepsilon^{(k)})\|_{\mathbb{C}^{rm}}\leq1$ for every $k$, there exists  a subsequence $(q(\varepsilon^{(k_p)}))_{p=1}^{\infty}\subset
(q(\varepsilon^{(k)}))_{k=1}^{\infty}$ and a nonzero vector $q(0)\in\mathbb{C}^{rm}$ such that $q(\varepsilon^{(k_p)})\to q(0)$ in $\mathbb{C}^{rm}$ as $p\to\infty$.

Thus, for every integer $p\geq1$, the function $z(\cdot,\varepsilon^{(k_p)})\in(C^{(n+r)})^{m}$ is a unique solution to the boundary-value problem
\begin{gather*}
L(\varepsilon^{(k_p)})\,z(t,\varepsilon^{(k_p)})=f(t,\varepsilon^{(k_p)}),
\quad a\leq t\leq b,\\
B(\varepsilon^{(k_p)})\,z(\cdot,\varepsilon^{(k_p)})=q(\varepsilon^{(k_p)}).
\end{gather*}
Since $f(\cdot,\varepsilon^{(k_p)})\to0$ in $(C^{(n)})^{m}$ and $q(\varepsilon^{(k_p)})\to q(0)$ in $\mathbb{C}^{rm}$ as $p\to\infty$, it follows from condition $(\ast\ast)$ of Basic Definition that
the function $z(\cdot,\varepsilon^{(k_p)})$ converges to the unique solution $z(\cdot,0)$ of the limiting boundary-value problem
\begin{equation*}
L(0)z(t,0)=0,\quad a\leq t\leq b,\qquad\mbox{and}\qquad
B(0)z(\cdot,0)=q(0).
\end{equation*}
This convergence holds in the space $(C^{(n+r)})^{m}$ as $\varepsilon\to0+$. But recall that $z(\cdot,\varepsilon^{(k_p)})\to0$ in the same space. Therefore $z(\cdot,0)\equiv0$, which contradicts the boundary condition $B(0)z(\cdot,0)=q(0)$ with $q(0)\neq0$. Thus, our assumption is wrong, and we have thereby proved that $\|B(\varepsilon)\|=O(1)$ as $\varepsilon\to0+$.

\emph{Step~$3$.} We can now prove that the boundary-value problem \eqref{6.syste}, \eqref{6.kue} satisfies Limit Condition~(II). Owing to \eqref{6.bound-norm-A} and Step~2, there exist numbers $\varkappa'>0$ and $\varepsilon'\in(0,\varepsilon_{1})$ such that $\|(L(\varepsilon),B(\varepsilon))\|\leq\varkappa'$ for each $\varepsilon\in[0,\varepsilon')$. Here, $\|(L(\varepsilon),B(\varepsilon))\|$ is the norm of the bounded operator \eqref{6.LBe}. We choose a function $z\in(C^{(n+r)})^{m}$ arbitrarily and then put $f(\cdot,\varepsilon):=L(\varepsilon)z$ and $q(\varepsilon):=B(\varepsilon)z$ for each $\varepsilon\in[0,\varepsilon')$. Thus,
\begin{equation*}
z=(L(\varepsilon),B(\varepsilon))^{-1}
(f(\cdot,\varepsilon),q(\varepsilon))\quad\mbox{for every}\quad
\varepsilon\in[0,\varepsilon').
\end{equation*}
Here, of course, $(L(\varepsilon),B(\varepsilon))^{-1}$ denotes the inverse of \eqref{6.LBe}. (The operator \eqref{6.LBe} is invertible due to condition ($\ast$) of Basic Definition.) Then, whenever $0<\varepsilon<\varepsilon'$, we have
\begin{align*}
&\bigl\|B(\varepsilon)z-B(0)z\bigr\|_{\mathbb{C}^{rm}}\leq
\bigl\|(f(\cdot,\varepsilon),q(\varepsilon))-
(f(\cdot,0),q(0))\bigr\|_{(C^{n})^{m}\times\mathbb{C}^{rm}}
\\
 & \qquad =\bigl\|(L(\varepsilon),B(\varepsilon))
(L(\varepsilon),B(\varepsilon))^{-1}
\bigl((f(\cdot,\varepsilon),q(\varepsilon))-
(f(\cdot,0),q(0))\bigr)\bigr\|_{(C^{(n)})^{m}\times\mathbb{C}^{rm}}
\\
&  \qquad
\leq\varkappa'\,\bigl\|(L(\varepsilon),B(\varepsilon))^{-1}
\bigl((f(\cdot,\varepsilon),q(\varepsilon))-
(f(\cdot,0),q(0))\bigr)\bigr\|_{(n+r)}
\\
& \qquad  =\varkappa'\,\bigl\|(L(0),B(0))^{-1}(f(\cdot,0),q(0))-
(L(\varepsilon),B(\varepsilon))^{-1}(f(\cdot,0),q(0))\bigr\|_{(n+r)}
\to0
\end{align*}
as $\varepsilon\to0+$ due to condition ($\ast\ast$) of Basic Definition. Since $y\in(C^{(n+r)})^{m}$ is arbitrary, we have proved that the boundary-value problem \eqref{6.syste}, \eqref{6.kue} satisfies Limit Condition~(II).
\end{proof}

\begin{proof}[Proof of Estimation Theorem]
We will first prove the left-hand part of~\eqref{6.bound}. Limit Conditions (I) and (II) yield the strong convergence of $(L(\varepsilon),B(\varepsilon))$ to $(L(0),B(0))$ as $\varepsilon\to0+$. Recall that we consider $(L(\varepsilon),B(\varepsilon))$, with $0\leq\varepsilon<\varepsilon_{0}$, as a bounded operator from
$(C^{(n+r)})^{m}$ to $(C^{(n)})^{m}\times\mathbb{C}^{rm}$. Hence, there exist numbers $\varkappa'>0$ and $\varepsilon'\in(0,\varepsilon_{1})$ such that the operator norm $\|(L(\varepsilon),B(\varepsilon))\|\leq\varkappa'$ for each $\varepsilon\in[0,\varepsilon')$. Indeed, if the contrary were true, we would obtain a sequence of positive numbers $(\varepsilon^{(k)})_{k=1}^{\infty}$ such that $\varepsilon^{(k)}\to0$ and $\|(L(\varepsilon^{(k)}),B(\varepsilon^{(k)}))\|\to\infty$ as $k\to\infty$, which would contradict the above strong convergence in view of the Banach-Steinhaus theorem. We see now that the left-hand part of the two-sided estimate \eqref{6.bound} holds for each $\varepsilon\in(0,\varepsilon')$, with $\varkappa_{1}:=1/\varkappa'$.

Let us prove the right-hand part of this estimate. According to Main Theorem, the boundary-value problem \eqref{6.syste}, \eqref{6.kue} satisfies Basic Definition. Therefore operator \eqref{6.LBe} is invertible for every $\varepsilon\in[0,\varepsilon_1)$, and, moreover, its inverse $(L(\varepsilon),B(\varepsilon))^{-1}$ converges strongly to $(L(0),B(0))^{-1}$ as $\varepsilon\to0+$. Indeed, for arbitrary $f\in(C^{n})^{m}$ and $q\in\mathbb{C}^{rm}$, we conclude by condition $(\ast\ast)$ of Basic Definition that
\begin{equation*}
(L(\varepsilon),B(\varepsilon))^{-1}(f,q)=:z(\cdot,\varepsilon)\to
z(\cdot,0):=(L(0),B(0))^{-1}(f,q)
\end{equation*}
in $(C^{(n+r)})^{m}$ as $\varepsilon\to0+$. Hence, there exists  positive numbers $\varepsilon_{2}<\varepsilon'$ and $\varkappa_{2}$ such that the norm of the inverse operator $\|(L(\varepsilon),B(\varepsilon))^{-1}\|\leq\varkappa_{2}$ for each $\varepsilon\in[0,\varepsilon_{2})$.  This follows from the Banach-Steinhaus theorem in a way analogous to that used above. We now see that the right-hand part of \eqref{6.bound} is true for  $\varepsilon\in(0,\varepsilon_{2})$.
\end{proof}

\section{Application}

Let us give an application of Main Theorem to multipoint
boundary-value problems. For the parameter-dependent differential
system \eqref{6.syste}, we pose the multipoint boundary condition
\begin{equation}\label{6.mp.kue}
B(\varepsilon)z(\cdot,\varepsilon)\equiv
\sum_{j=0}^{p}\sum_{k=1}^{\omega_j}
\sum_{l=0}^{n+r}{\alpha_{j,k}^{(l)}(\varepsilon)
z^{(l)}(t_{j,k}(\varepsilon),\varepsilon)}=q(\varepsilon).
\end{equation}
Here, the integers $p\geq1$ and $\omega_j\geq1$, matrixes
$\alpha_{j,k}^{(l)}(\varepsilon)\in\mathbb{C}^{rm\times m}$, and
points  $t_{j,k}(\varepsilon)\in[a,b]$ are arbitrarily chosen for
all admissible values of $j$, $k$, $l$, and $\varepsilon$. As
above, $q(\varepsilon)\in\mathbb{C}^{rm}$. In~\eqref{6.mp.kue}, we
use the summation over the two indexes $j$ and $k$ instead of one
index because the points $t_{j,k}(\varepsilon)$, with $j\geq1$,
will be supposed to converge to a certain point $t_{j}$ as
$\varepsilon\to0+$, whereas the points $t_{0,k}(\varepsilon)$ will
not be supposed to have a limit as $\varepsilon\to0+$. So, we
divide the set of all points $t_{j,k}(\varepsilon)$ into $p+1$
sets $\{t_{j,k}(\varepsilon):k=1,\ldots,\omega_j\}$, with
$j=0,1,\ldots,p$, in accordance with their behavior as
$\varepsilon\to0+$.

In this connection, we consider the following limiting
boundary-value problem:
\begin{gather}
L(0)z(t,0)=f(\cdot,0),\quad a\leq t\leq b,\label{6.mp.L0}\\
B(0)z(\cdot,0)\equiv\sum_{j=1}^{p}
\sum_{l=0}^{n+r}{\alpha_{j}^{(l)}z^{(l)}(t_{j},0)}=q(0).
\label{6.mp.lkue}
\end{gather}
Here, $\alpha_{j}^{(l)}\in\mathbb{C}^{rm\times m}$,
$t_{j}\in[a,b]$ for all admissible values of $j$ and $l$, and as
above, $q(0)\in\mathbb{C}^{rm}$.

It is evident that, for every $\varepsilon\in[0,\varepsilon_0)$,
the linear mapping $z(\cdot,\varepsilon)\mapsto
B(\varepsilon)z(\cdot,\varepsilon)$ is a continuous operator from
$(C^{(n+r)})^m$ to $\mathbb{C}^{rm}$. Thus, the multipoint
boundary-value problem \eqref{6.syste}, \eqref{6.mp.kue} is
generic with respect to the space $(C^{(n+r)})^m$ for every
$\varepsilon\in(0,\varepsilon_0)$ as well as the the limiting
problem \eqref{6.mp.L0}, \eqref{6.mp.lkue}.

For these boundary-value problems, we can give constructive
sufficient conditions for the convergence $B(\varepsilon)z\to
B(0)z$ in $\mathbb{C}^{rm}$ for every $z\in(C^{(n+r)})^m$. Note
that it is scarcely possible to use the system of conditions
(2a)--(2d) for the verification of this convergence because it is
difficult in practice to find the function
$\Phi(\cdot,\varepsilon)$ in the canonical
representation~\eqref{Be-description} of the multipoint boundary
operator $B(\varepsilon)$.

\medskip

\noindent \textbf{Theorem A1.} \it  Suppose that the
boundary-value problem \eqref{6.syste}, \eqref{6.mp.kue} satisfies
the following conditions as $\varepsilon \to 0+$:
\begin{itemize}
    \item[\textup{(d1)}] $t_{j,k}(\varepsilon)\to t_{j}$ for all $j\in\{1,\ldots,p\}$ and $k\in\{1,\ldots,\omega_j\}$;
    \item[\textup{(d2)}] $\sum_{k=1}^{\omega_j}
        \alpha_{j,k}^{(l)}(\varepsilon)\to\alpha_{j}^{(l)}$ for all $j\in\{1,\ldots,p\}$ and $l\in\{0,\ldots,n+r\}$;
    \item[\textup{(d3)}] $\alpha_{j,k}^{(n+r)}(\varepsilon)=O(1)$ for all $j\in\{1,\ldots,p\}$ and $k\in\{1,\ldots,\omega_j\}$;
    \item[\textup{(d4)}] $\|\alpha_{j,k}^{(l)}(\varepsilon)\|_{\mathbb{C}^{rm\times m}}
        \!\cdot\!|t_{j,k}(\varepsilon)-t_{j}|\to0$ for all $j\in\{1,\ldots,p\}$, $k\in\{1,\ldots,\omega_j\}$, and $l\in\{0,\ldots,n+r-1\}$;
    \item[\textup{(d5)}] $\alpha_{0,k}^{(l)}(\varepsilon)\to0$ for all $k\in\{1,\ldots,\omega_0\}$ and $l\in\{0,\cdots,n+r\}$.
\end{itemize}
\noindent Then this problem satisfies Limit Condition
\textup{(II)}.\rm

\medskip

Owing to Main Theorem and Theorem A1, we immediately obtain the
following result:

\medskip

\noindent \textbf{Theorem A2.} \it Assume that the multipoint
boundary-value problem \eqref{6.syste}, \eqref{6.mp.kue} satisfies
Limit Condition \textup{(I)} and conditions
\textup{(d1)}--\textup{(d5)} and that the limiting problem
\eqref{6.mp.L0}, \eqref{6.mp.lkue} satisfies Condition~$(0)$. Then
the solution to the problem \eqref{6.syste}, \eqref{6.mp.kue} is
continuous in the parameter $\varepsilon$ at $\varepsilon=0$. \rm

\begin{proof}[Proof of Theorem \textup{A.1}] Let $\|\cdot\|$ denote the norm of a number vector or matrix, this norm being the sum of the absolute values of all their elements. For an arbitrary function $z\in(C^{(n+r)})^m$ and sufficiently small $\varepsilon>0$, we have the inequality
\begin{equation}\label{6.mp.eq1}
\begin{aligned}
\|B(\varepsilon)z-B(0)z\|\leq &\sum_{k=1}^{\omega_0}
\sum_{l=0}^{n+r} \|\alpha_{0,k}^{(l)}(\varepsilon)\|\cdot
\|z^{(l)}(t_{0,k}(\varepsilon))\|\\
&+\sum_{j=1}^{p}\sum_{l=0}^{n+r}\,
\left\|\sum_{k=1}^{\omega_j}\alpha_{j,k}^{(l)}(\varepsilon)
z^{(l)}(t_{j,k}(\varepsilon))-\alpha_{j}^{(l)}z^{(l)}(t_{j})\right\|.
\end{aligned}
\end{equation}
Here, by condition (d5), we can write
\begin{equation}\label{6.mp.eq22}
\begin{gathered}
\|\alpha_{0,k}^{(l)}(\varepsilon)\|\cdot
\|z^{(l)}(t_{0,k}(\varepsilon))\|\leq
\|\alpha_{0,k}^{(l)}(\varepsilon)\|\cdot \|z\|_{(n+r)}\to0
\end{gathered}
\end{equation}
for all admissible values of $k$ and $l$. This and the other
limits in the proof are considered as $\varepsilon\to0+$. Besides,
\begin{equation}\label{6.mp.eq23}
\begin{aligned}
&\left\|\sum_{k=1}^{\omega_j}
\alpha_{j,k}^{(l)}(\varepsilon)z^{(l)}(t_{j,k}(\varepsilon))
-\alpha_{j}^{(l)}z^{(l)}(t_{j})\right\|\\
&\leq\left\|\sum_{k=1}^{\omega_j}
\alpha_{j,k}^{(l)}(\varepsilon)\cdot
\left(z^{(l)}(t_{j,k}(\varepsilon))-z^{(l)}(t_{j})\right)\right\|
+\left\|\left(\sum_{k=1}^{\omega_j}
\alpha_{j,k}^{(l)}(\varepsilon)-\alpha_{j}^{(l)}\right)\cdot
z^{(l)}(t_{j})\right\|\\ &\leq\sum_{k=1}^{\omega_j}
\|\alpha_{j,k}^{(l)}(\varepsilon)\|\cdot
\|z^{(l)}(t_{j,k}(\varepsilon))-z^{(l)}(t_{j})\|
+\left\|\sum_{k=1}^{\omega_j}
\alpha_{j,k}^{(l)}(\varepsilon)-\alpha_{j}^{(l)}\right\|\cdot
\|z\|_{(n+r)}.
\end{aligned}
\end{equation}
Here, according to condition~(d2), we have
\begin{equation}\label{6.mp.eq24}
\left\|\sum_{k=1}^{\omega_j}
\alpha_{j,k}^{(l)}(\varepsilon)-\alpha_{j}^{(l)}\right\|\cdot
\|z\|_{(n+r)}\to0.
\end{equation}
Moreover,
\begin{equation}\label{6.mp.eq2}
\sum_{k=1}^{\omega_j}
\|\alpha_{j,k}^{(l)}(\varepsilon)\|\cdot\|
z^{(l)}(t_{j,k}(\varepsilon))-z^{(l)}(t_{j})\,\|\to0.
\end{equation}
Indeed, if $l=n+r$, then formula \eqref{6.mp.eq2} is a direct
consequence of conditions (d1), (d3) and the continuity of
$z^{(l)}$. If $l \leq n+r-1$, then this formula is a consequence
of the Lagrange theorem of the mean and condition~(d4), namely
\begin{equation*}
\sum_{k=1}^{\omega_j}
\|\alpha_{j,k}^{(l)}(\varepsilon)\|\cdot\|
z^{(l)}(t_{j,k}(\varepsilon))-z^{(l)}(t_{j})\,\| \leq\|z\|_{(n+r)}
\sum_{k=1}^{\omega_j}
\|\alpha_{j,k}^{(l)}(\varepsilon)\|\cdot
|t_{j,k}(\varepsilon)-t_{j}|\to0.
\end{equation*}

According to \eqref{6.mp.eq23}--\eqref{6.mp.eq2}, we have the
convergence
\begin{equation}\label{6.mp.eq25}
\left\|\sum_{k=1}^{\omega_j}
\alpha_{j,k}^{(l)}(\varepsilon)z^{(l)}(t_{j,k}(\varepsilon))
-\alpha_{j}^{(l)}z^{(l)}(t_{j})\right\|\to0.
\end{equation}
Now formulas \eqref{6.mp.eq1}, \eqref{6.mp.eq22}, and
\eqref{6.mp.eq25} imply that $\|B(\varepsilon)z-B(0)z\|\to0$, i.e.
Limit Condition~(II) is satisfied.
\end{proof}

\section{Concluding remarks}

It is easy to check that conditions $(\ast)$ and $(\ast\ast)$ in Basic Definition are separately equivalent to the corresponding ones:
\begin{itemize}
\item [$(\star)$] there exists a positive number $\varepsilon_1<\varepsilon_0$ such that the operator \eqref{6.LBe} is invertible for every $\varepsilon\in[0,\varepsilon_1)$;
\item [$(\star\star)$] the operator $(L(\varepsilon),B(\varepsilon))^{-1}$, which is inverse of \eqref{6.LBe}, converges to $(L(0),B(0))^{-1}$ as $\varepsilon\to\infty$ in the strong operator topology.
\end{itemize}

It is also not difficult to prove that the system of Limit Conditions (I) and (II) is equivalent to

\medskip

\noindent\textbf{Limit Condition (IV).} The operator $(L(\varepsilon),B(\varepsilon))$ converges to $(L(0),B(0))$ as $\varepsilon\to\infty$ in the strong operator topology.

\medskip

So, Main Theorem implies that, the equivalence
\begin{equation*}
\bigl(\mbox{Limit Condition (IV)}\bigr)\;\Leftrightarrow\; \bigl((\star)\wedge(\star\star)\bigr)
\end{equation*}
holds true under Condition~(0).

This is a surprising result at first sight. Indeed, if $X$ and $Y$ are arbitrary infinite-dimensional Banach spaces with Schauder basis, then the set of all irreversible operators is sequentially dense in $\mathcal{L}(X,Y)$ in the strong operator topology; here, $\mathcal{L}(X,Y)$ stands for the linear space of all continuous linear operators from $X$ to $Y$. Moreover, the mapping $\mathrm{inv}:T\mapsto T^{-1}$ given on the set of all invertible operators $T\in\mathcal{L}(X,Y)$ is everywhere discontinuous in this topology. We will prove these propositions in Appendix.

\section*{Appendix}

Let $X$ and $Y$ are complex or real Banach spaces. Assume that at least one of the spaces $X$ and $Y$ has Schauder basis. Specifically, we can put $X:=(C^{(n+r)})^{m}$ and $Y:=(C^{(n)})^{m}\times\mathbb{C}^{rm}$.

\begin{proposition}\label{prop1}
The set of all finite-dimensional operators from $\mathcal{L}(X,Y)$ is sequentially dense in $\mathcal{L}(X,Y)$ in the strong operator topology.
\end{proposition}

Since every finite-dimensional operator from $\mathcal{L}(X,Y)$ is irreversible, this proposition implies that the set of all irreversible operators from $\mathcal{L}(X,Y)$ is sequentially dense in $\mathcal{L}(X,Y)$ in the strong operator topology.

\begin{proof}[Proof of Proposition $\ref{prop1}$]
We apart examine the cases where $X$ or $Y$ has Schauder basis.

Consider first the case where $X$ has a certain Schauder basis $(e_{k})_{k=1}^{\infty}$. Then every vector $x\in X$ admits the unique representation
\begin{equation}\label{append-f1}
x=\sum_{k=1}^{\infty}e_{k}^{\ast}(x)\,e_{k},
\end{equation}
where each $e_{k}^{\ast}$ is a certain continuous linear functional on $X$. Using this basis, we define the operator of the $n$-th partial sum \begin{equation*}
S_{n}:x\mapsto\sum_{k=1}^{n}e_{k}^{\ast}(x)\,e_{k},\quad\mbox{with}\quad x\in X.
\end{equation*}
It is evident that $S_{n}\rightarrow I_{X}$ as $n\to\infty$ in the strong operator topology in $\mathcal{L}(X,X)$; here, as usual, $I_{X}$ denotes the identical operator on $X$. Now, given an operator $T\in\mathcal{L}(X,Y)$, we put $T_{n}:=TS_{n}$ for each integer $n\geq1$. Then every operator $T_{n}$ is finite-dimensional, and $T_{n}\to T$ as $n\to\infty$ in the strong operator topology in $\mathcal{L}(X,Y)$.

The case where $Y$ has Schauder basis is examined similarly. Namely, let $(\tilde{e}_{k})_{k=1}^{\infty}$ be a Schauder basis of $Y$, and let $\tilde{e}_{k}^{\ast}(y)$ be the $k$-th coordinate of $y\in Y$ in this basis. For every integer $n\geq1$, we define the operator
\begin{equation*}
\widetilde{S}_{n}:y\mapsto
\sum_{k=1}^{n}\tilde{e}_{k}^{\ast}(y)\,\tilde{e}_{k},
\quad\mbox{with}\quad y\in Y.
\end{equation*}
Given an operator $T\in\mathcal{L}(X,Y)$, we put $\widetilde{T}_{n}:=\widetilde{S}_{n}T$ and conclude that $\widetilde{T}_{n}$ is finite-dimensional and that $\widetilde{T}_{n}\to T$ as $n\to\infty$ in the strong operator topology in $\mathcal{L}(X,Y)$.
\end{proof}

\begin{proposition}\label{prop2}
Assume that the spaces $X$ and $Y$ are isomorphic. Then the mapping $\mathrm{inv}:T\mapsto T^{-1}$ given on the set of all invertible operators $T\in\mathcal{L}(X,Y)$ is everywhere discontinuous in the strong operator topology.
\end{proposition}

The assumption made in this proposition is natural; indeed, if $X$ and $Y$ are not isomorphic, then there are not any invertible operators in $\mathcal{L}(X,Y)$.

\begin{proof}[Proof of Proposition $\ref{prop2}$]
Owing to the Banach--Steinhaus theorem, it is sufficient to prove that, for an arbitrary invertible operator $T\in\mathcal{L}(X,Y)$, there exists a sequence of invertible operators $(T_{n})_{n=1}^{\infty}\subset\mathcal{L}(X,Y)$ such that $T_{n}\to T$ as $n\to\infty$ in the strong operator topology in $\mathcal{L}(X,Y)$ and that the sequence $(\|T_{n}^{-1}\|)_{n=1}^{\infty}$ is not bounded. Here, of course, $\|T_{n}^{-1}\|$ stands for the norm of the inverse operator $T_{n}^{-1}:Y\to X$. Without loss of generality, we can assume that $Y=X$ and $T=I_{X}$. Indeed, since the Banach spaces $X$ and $Y$ are isomorphic, the general case is trivially reduced to the case assumed.

Let us use the reasoning given in the proof of Proposition~\ref{prop1}. For every integer $n\geq1$, we consider the operator
\begin{equation}\label{append-f2}
I_{n}:x\mapsto\sum_{k=1}^{n-1}e_{k}^{\ast}(x)\,e_{k}+
\frac{1}{n}\,e_{n}^{\ast}(x)\,e_{n}+
\sum_{k=n+1}^{\infty}e_{k}^{\ast}(x)\,e_{k},\quad\mbox{with}\quad
x\in X.
\end{equation}
Every operator $I_{n}$ belongs to $\mathcal{L}(X,Y)$ and is invertible; moreover,
\begin{equation}\label{append-f3}
I_{n}^{-1}:x\mapsto\sum_{k=1}^{n-1}e_{k}^{\ast}(x)\,e_{k}+
n\,e_{n}^{\ast}(x)\,e_{n}+
\sum_{k=n+1}^{\infty}e_{k}^{\ast}(x)\,e_{k},\quad\mbox{with}\quad
x\in X.
\end{equation}
It follows from \eqref{append-f1} and \eqref{append-f2} that $I_{n}\to I_{X}$ as $n\to\infty$ in the strong operator topology in $\mathcal{L}(X,Y)$. However, the sequence  $(\|I_{n}^{-1}\|)_{n=1}^{\infty}$ is not bounded because
\begin{equation*}
\|I_{n}^{-1}\|\geq\frac{\|I_{n}^{-1}e_{n}\|}{\|e_{n}\|_{X}}=n
\quad\mbox{for each integer}\quad n\geq1
\end{equation*}
due to \eqref{append-f3}.
\end{proof}

Note that we cannot replace the strong operator topology with the uniform one in Propositions \ref{prop1} and~\ref{prop2}. As is known, the mapping $\mathrm{inv}$ is everywhere continuous in the uniform operator topology, and the closure of the set of all finite-dimensional operators in this topology contains only compact (hence, irreversible) operators.



\end{document}